\documentclass[12pt]{article}
\usepackage{amssymb,amsfonts,amsmath,latexsym,epsf,tikz,url}
\usepackage[usenames,dvipsnames]{pstricks}
\usepackage{pstricks-add}
\usepackage{epsfig}
\usepackage{pst-grad} 
\usepackage{pst-plot} 
\usepackage[space]{grffile} 
\usepackage{etoolbox} 
\makeatletter 
\patchcmd\Gread@eps{\@inputcheck#1 }{\@inputcheck"#1"\relax}{}{}
\makeatother

\newtheorem{theorem}{Theorem}[section]
\newtheorem{proposition}[theorem]{Proposition}
\newtheorem{observation}[theorem]{Observation}

\newtheorem{corollary}[theorem]{Corollary}

\newcommand{\qed}{\hfill $\blacksquare$\medskip}

\begin{document}

\title{ Independent domination bondage number in graphs}

\author{ M. Mehraban  and S. Alikhani$^{}$\footnote{Corresponding author, \texttt{https://orcid.org/0000-0002-1801-203X}} 
}

\date{\today}

\maketitle

\begin{center}
	
Department of Mathematical Sciences, Yazd University, 89195-741, Yazd, Iran

	\bigskip
	{\tt  Mazharmehraban2020@gmail.com, ~~alikhani@yazd.ac.ir  
		}

\end{center}

\begin{abstract}
  A non-empty set $S\subseteq V (G)$ of the simple graph $G=(V(G),E(G))$ is an independent dominating
 set of $G$ if every vertex not in $S$ is adjacent with some vertex in $S$ and the vertices of $S$ are pairwise non-adjacent.  The independent domination number of $G$, denoted by $\gamma_i(G)$, is the minimum size of all independent dominating sets of $G$. The independent domination bondage number of $G$ is the minimum number of edges whose removal changes the independent domination number of $G$.
 In this paper, we investigate properties of independent domination bondage number in graphs. In particular, we obtain several bounds and obtain  the independent domination bondage number  of some operations of two graphs. 
\end{abstract}

\noindent{\bf Keywords:} dominating set, independent domination number, independent domination bondage number, operation.

\medskip
\noindent{\bf AMS Subj.\ Class.}:  05C05, 05C69.
\section{Introduction}
 Let $G=(V,E)$ be a simple graph with finite number of vertices. 
  The open neighborhood of a vertex $v\in V(G)$  is the set of vertices that are adjacent to $v$, but not including $v$ itself, $N(v) = \{u \in V(G) : uv \in E(G)\}$ and the closed neighborhood of a vertex $v\in V(G)$ is the open neighborhood of $v$ along with the vertex $v$ itself and is denoted as $N[v] = N(v) \cup \{v\}.$
   For a set $S\subseteq V(G)$, the open neighborhood of $S$ is $N(S)=\bigcup_{v\in S} N(v)$ and  the closed neighborhood of $S$
   is $N[S]=N(S)\cup S$. The private neighborhood $pn(v,S)$ of $v\in S$ is defined by $pn(v,S)=N(v)-N(S-\{v\})$, equivalently, $pn(v,S)=\{u\in V| N(u)\cap S=\{v\}\}$.
   The degree of a vertex $v$  denoted as $deg(v)$, is the number of edges incident to that vertex which is equal to $|N(v)|$.  A leaf of a tree is a vertex of degree 1. 
   
   A dominating set of a graph $G$ is a set S of vertices of $G$ such that every vertex
   not in $S$ is adjacent to a vertex in $S$.
   The domination number of $G$, denoted by $\gamma(G)$, is the minimum size of a dominating set. A domination-critical vertex in a graph $G$ is a vertex whose removal decreases the domination number (\cite{8,9}).  It is easy to observe that for any graph $G$, for every edge $e\not\in E(G)$. we have $\gamma(G)-1\leq \gamma(G+e)\leq \gamma(G)$.
   A graph is said to be domination stable, or $\gamma$-stable, if $\gamma(G) =\gamma(G+e)$ for
   every edge $e$ in the complement $G^c$ of $G$ (\cite{15,11}).
   Stability of some kind of domination number can be found in \cite{4,5,12,wcds,14,16}. 
 
A set is independent if no two vertices in it are adjacent. An independent dominating set of $G$ is a set that is
both dominating and independent in $G$. The independent domination number of $G$, denoted by $\gamma_i(G)$, is the minimum size of
an independent dominating set. The independence number of $G$, denoted by $\alpha(G)$, is the maximum size of an independent set
in $G$. It follows immediately that $\gamma(G) \leq \gamma_i(G)\leq \alpha(G)$.
A dominating set of $G$ of size $\gamma(G)$ is called a $\gamma$-set, while an independent dominating set of $G$ of size $\gamma_i(G)$ is called an $i$-set.
A graph $G$ is independent domination critical, or $\gamma_i$-critical if $\gamma_i(G-v)\neq \gamma_i(G)$ for every $v\in V(G)$.
In $\gamma_i$-critical graph, no single vertex can be removed from the minimum dominating set without increasing the domination number of the graph.  The independent domination and independent domination critical graphs have been reported in \cite{4,16,7,13}.
This paper is organized as follows. 

\medskip

In the next section, we introduce the  independent domination bondage number (id-bondage number) of graphs and  present some of its properties.  In Section 3, we obtain  some bounds for the independent domination bondage number of graphs. In Section 4, we study the independent domination bondage number of some operations of two graphs. 

\section{ id-bondage number of certain graphs}
In this section, we first consider the  independent domination bondage number (id-bondage number) of graphs. 
The independent domination bondage number of a graph $G$, denoted as 
$b_{id}(G)$, is defined as the minimum number of edges that need to be removed from the graph such that the independent domination number of the graph changes. Formally, the independent domination bondage number of a graph $G$ is given by $b_{id}(G)=\min\{|E'|:\gamma_i(G-E')\neq\gamma_i(G)\}$.

Now we  determine the independent domination bondage number for some classes of graphs. To aid our discussion, we first state some straightforward observations as follows.

The following observation is easy to obtain: 
\begin{observation}
	\begin{enumerate}
		\item [(i)] 
If $G=K_{1,n}$ is  a star, then $b_{id}(K_{1,n})=1$.
\item[(ii)]
If $K_{m,n}$ is  a bipartite complete graph, then $b_{id}(K_{m,n})=\min\{m,n\}$.
	\end{enumerate}
\end{observation}

The friendship graph $F_n$ is a graph that can be constructed by the coalescence of $n$
copies of the cycle graph $C_3$ of length $3$ with a common vertex. The Friendship Theorem of Paul Erd\"{o}s,
Alfred R\'{e}nyi and Vera T. S\'{o}s \cite{erdos}, states that graphs with the property that every two vertices have
exactly one neighbour in common are exactly the friendship graphs. 
Let $n$ and $q\geq 3$ be any positive integer and  $F_{q,n}$ be the {\em generalized friendship graph}  formed by a collection of $n$ cycles (all of order $q$), meeting at a common vertex (see Figure~\ref{Friendship}).

  \begin{figure}[h!]
  	\centering
  	\includegraphics[height=3cm , width=10cm]{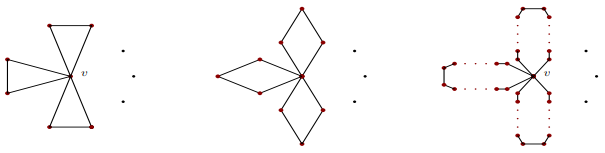}
  	\caption{ The flowers $F_n$, $F_{4,n}$ and $F_{q,n}$, respectively.}
  	\label{Friendship}
  \end{figure}

The $n$-book graph $(n\geq2)$ (Figure~\ref{Book}) is defined as the Cartesian product $K_{1,n}\square P_2$. We call every $C_4$ in the book graph $B_n$, a page of $B_n$. All pages in $B_n$ have a common side $v_1v_2$.   We shall compute the id-bondage number  of $B_n$. 
The following observation gives the independent domination number of  the friendship graph and the book graph. 

  \begin{figure}[h!]
  	\centering
  	\includegraphics[height=3cm , width=8cm]{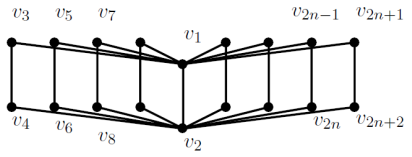}
  	\caption{ Book graph $B_n$.}
  	\label{Book}
  \end{figure}

\begin{observation}
	\begin{enumerate} 
		\item[(i)] 
		If $F_n$ is a friendship graph, then $\gamma_i(F_n)=1$.
		
		\item[(ii)] $\gamma_i(F_{4,n})=\gamma_i(F_{5,n})=\gamma_i(F_{6,n})=n+1$
		
		\item[(iii)] 
		If $B_n$ is a book graph, then $\gamma_i(B_n)=n$.
	\end{enumerate} 
\end{observation}
Now we obtain the id-bondage number of the friendship and the book graphs: 
\begin{proposition}
	\begin{enumerate} 
		\item[(i)] 
		If $F_n$ is a friendship graph, then $b_{id}(F_n)=2$.
		
		\item[(ii)] For $k=4,5,6$, $b_{id}(F_{k,n})=3$.
		
		\item[(iii)] 
		$b_{id}(B_n)=3$. 
		
	\end{enumerate} 
\end{proposition}

\begin{proof}
	\begin{enumerate}
	\item[(i)] 
		By removing two edges incident to central vertex, we have:
	\begin{equation*}
	\begin{split}
	 \gamma_i(F_n-\{e_1,e_2\}) &=\gamma_i(K_2\cup F_{n-1})\\
	&=\gamma_i(K_2)+\gamma_i(F_{n-1})\\
	&=1+1=2\neq \gamma_i(F_n).\\  
	\end{split}
	\end{equation*}
	So we have the result. 
	\item[(ii)] 
   By removing three edges from $F_{k,n}$, for $k=4,5,6$, we have:
    	\begin{equation*}
    	\begin{split}
    	\gamma_i(F_{k,n}-\{e_1,e_2,e_3\}) &=\gamma_i(F_{k,n-1}\cup2K_2),\\
    	&=\gamma_i(F_{k,n-1})+\gamma_i(2K_2)\\
    	&=n+2\neq n+1.\\
    	\end{split}
    	\end{equation*}
 Thus, we have the result.
		\item[(iii)] 
		By removing three consecutive edges $\{e_1,e_2,e_3\}$ from $B_n$, we have: 
			\begin{equation*}
			\begin{split}
		\gamma_i(B_n-\{e_1,e_2,e_3\})&=\gamma_i(2K_1\cup B_{n-1}),\\ 
			&=\gamma_i(2K_1)+\gamma_i(B_{n-1})\\
			&=2+n-1=n+1\neq n.\\
			\end{split}
			\end{equation*}
	 Thus, we have the result.\qed
	\end{enumerate}
\end{proof}

\begin{theorem}
There exist graphs $G$ and $H$ with the same independent domination number such that $|b_{id}(G)-b_{id}(H)|$ can be arbitrary large.

\end{theorem} 
\begin{proof}
	Let $G=K_n$ and $H=K_{1,n-1}$. For $n\geq 3$,  $\gamma_i(G)=\gamma_i(H)=1$, $b_{id}(G)=n-1$,  and $b_{id}(H)=1$. We have $|b_{id}(G)-b_{id}(H)|=n-2$, which grows  arbitrarily large as $n\rightarrow \infty$. Thus, the result follows. \qed
\end{proof}

Before investigating the independent domination bondage number of paths ans cycles, we state  the following observation.

\begin{observation}{\rm \cite{7}}\label{pathcy}
	For the path and cycle, $\gamma_i(P_n)=\gamma_i(C_n)= \lceil\frac{n}{3}\rceil$.
\end{observation}
The following result establishes an upper bound on the independent domination bondage number in terms of $\delta(G)$.

\begin{observation}\label{delta} 
	For any graph $G$,  $b_{id}(G) \leq  \delta(G)+1$.
\end{observation}

Now, we  determine the id-bondage of paths.
\begin{proposition}
    If $P_{n}$ is a path of order $n$, then
 $b_{id}(P_n)=\left\{
 \begin{array}{cc}
 
 2   &\quad n\equiv 1 ~(\mbox{mod } 3)\\
 1    &\quad otherwise\\
 \end{array}\right.
. $
\end{proposition}

\begin{proof} We consider  the following cases:  
	
	\noindent{\bf Case 1.} $ n\equiv 1 ~(\mbox{mod } 3)$
	
	For $n=3k+1$, for some integer $k\geq 1$. By removing two specific edges $\{e_1,e_2\}$ of $P_n$, we have $\gamma_i(P_n-\{e_1,e_2\})=\gamma_i(2K_1\cup P_{n-2})=2+\gamma_i(P_{n-2})=k+2\neq k+1=\gamma_i(P_n)$. Thus, $b_{id}(P_n)=2$.
	
	\noindent{\bf Case 2.} $n\not\equiv 1 ~(\mbox{mod } 3)$ 
	
	For $n=3k$ or $n=3k+2$, removing any single edge suffices to increase $\gamma_i(P_n)$, so we have $\gamma_i(P_n-e_1)>\gamma_i(P_n)$. Thus, $b_{id}(P_n)=1$.
	\qed
\end{proof}

Next, we determine the $id$-bondage number of cycles.
\begin{proposition}
	
	If $C_{n}$ is  a cycle of order $n$, then
	$b_{id}(C_n)=\left\{
	\begin{array}{cc}
	3  &\quad n\equiv 1 ~(\mbox{mod } 3)\\
	2    &\quad otherwise\\
	\end{array}\right.
	.$
\end{proposition}
\begin{proof} We consider the following cases:  
	
	\noindent{\bf Case 1.} $n\not\equiv 1 ~(\mbox{mod } 3)$ 
	
		For $n=3k$ or $n=3k+2$,  removing two specific edges suffices to change $\gamma_i(C_n)$. We have $\gamma_i(C_n-\{e_1,e_2\})=\gamma_i(K_1\cup P_{n-1})=1+\gamma_i(P_{n-1})\neq\gamma_i(C_n)$. Thus, $b_{id}(C_n)=2$. 

	\noindent{\bf Case 2.} $ n\equiv 1 ~(\mbox{mod } 3)$
	
	For $n=3k+1$, removing three specific edges change $\gamma_i(C_n)$. We have $\gamma_i(C_n-\{e_1,e_2,e_3\})=\gamma_i(2K_1\cup P_{n-3})=2+\gamma_i(P_{n-3})\neq\gamma_i(C_n)$. Thus, $b_{id}(C_n)=3$. \qed
\end{proof}

The following observation determines the independent domination number of the complete graph $K_n$.
\begin{observation}\label{delta2} 
	If $K_n$ is a complete graph, then $\gamma_i({K_n})=1$.
\end{observation}

Now, we will compute the id-bondage number of the complete graph $K_n$.
\begin{proposition}
	If $K_n$ is a complete graph of order $n$, then $b_{id}({K_n})=n-1$.
\end{proposition}
\begin{proof}
	We know that $\gamma_i(K_n)=1$. Removing at least $n-1$ edges to change  the independent domination number of $K_n$,  
		\begin{equation*}
		\begin{split}
		\gamma_i(K_n-(n-1)e)&=\gamma_i(K_i\cup K_{n-1})\\ 
		&=\gamma_i(K_1)+\gamma_i(K_{n-1})\\
		&=1+1=2\neq \gamma_i(K_n)=1.\\
		\end{split}
		\end{equation*}
 Thus, $b_{id}(K_n)=n-1$. Hence, we have the result. \qed
\end{proof}

\bigskip
Next, we turn our attention to cactus graphs.

A {cactus graph} is a connected graph in which no edge lies in more than one cycle. Consequently, each block of a cactus graph is either an edge or a cycle. If all blocks of a cactus $G$ are cycles of the same size $m$, the cactus is $m$-uniform.

A {triangular cactus} is a graph whose blocks are triangles, i.e., a $3$-uniform cactus. A vertex shared by two or more triangles is called a cut-vertex. If each triangle of a triangular cactus $G$ has at most two cut-vertices, and each cut-vertex is shared by exactly two triangles, we say that $G$ is a chain triangular cactus. The number of triangles in $G$ is called the length of the chain. A chain triangular cactus is (shown in Figure~\ref{Trangular}). Obviously a chain triangular cactus of length $n$ has $2n + 1$ vertices and $3n$ edges (see \cite{IJMC}). 
 \begin{figure}[h!] 
 	\centering
 	\includegraphics[height=1.5cm , width=8cm]{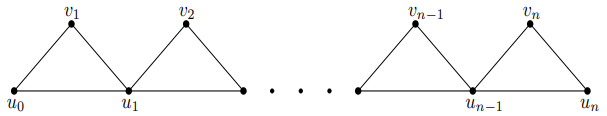}
 	\caption{The chain triangular cactus $T_n$.}
 	\label{Trangular}
 \end{figure}

Furthermore, any chain triangular cactus of length greater than one has exactly two triangles with only one cut-vertex.  Hence, we denote the chain triangular cactus of length $n$ by $T_n$.

By replacing triangles in the above definitions by cycles of length $4$ we obtain cacti whose every block is $C_4$. We call such cacti {square cacti}. A square cactus chain is shown in Figure~\ref{Q_n}. We see that the internal squares may differ in the way they connect to their neighbors. If their cut-vertices are adjacent, we say that such a square is an {ortho-square}; if the cut-vertices are not adjacent, we call the square a para-square. The set of all chain square cacti of length $n$ will be denoted by $Q_n$. The unique square cactus chain of length $n$ whose all internal squares are {para-squares} we denote by $Q_n$, while the unique ortho-chain will be denoted by $O_n$ (see Figure~\ref{O_n}).

  \begin{figure}[h!]
  	\centering
  	\includegraphics[height=2.3cm , width=8.3cm]{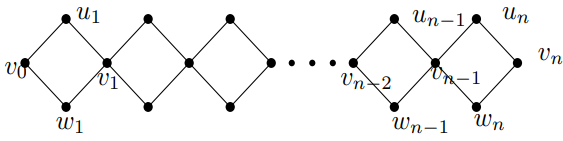}
  	\caption{para-chain square cactus graphs $Q_n$.}
  	\label{Q_n}
  \end{figure}

  \begin{figure}[h!]
  	\centering
  	\includegraphics[height=2.3cm , width=8.5cm]{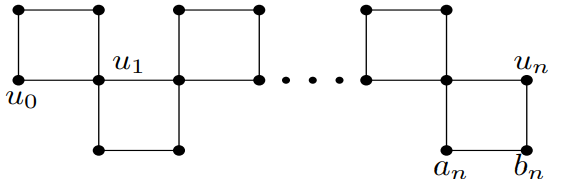}
  	\caption{ortho-chain square cactus $O_n$.}
  	\label{O_n}
  \end{figure}

The following observation gives the independent domination number of the chain-triangular cactus graph $T_n$. 

\begin{observation}
	If $T_n$ is a chain triangular cactus graph, then $\gamma_i(T_n)=n.$
\end{observation}
  We compute the id-bondage number of $T_n$.
\begin{proposition}
	If $T_n$ is a chain triangular cactus graph, then $b_{id}(T_n)=2$.
\end{proposition}
\begin{proof}
	By removing one edge from $T_n$, the independent domination number of $T_n$ dose not change.
	 However, if we remove two edges from $T_n$, the independent domination number of $T_n$ will change. Since $\gamma_i(T_n)=n$, we have $\gamma_i(T_n-\{e_1,e_2\})\neq\gamma_i(T_n)=n$. Thus, $b_{id}(T_n)=2$.\qed
\end{proof}

The following observation gives the independent domination number of the para-chain cactus graph $Q_n$ and the ortho-chain square cactus graph $O_n$. 

\begin{observation}
	For the para-chain cactus graph and ortho-chain square cactus graph, $\gamma_i(Q_n)=\gamma_i(O_n)=\lceil\frac{n}{2}\rceil$.
\end{observation}
We first compute the id-bondage number of the para-chain cactus graph $Q_n$.
\begin{proposition}\label{Q_n1} 
	If $Q_n$ is a para-chain cactus graph, then $b_{id}(Q_n)=2$.
\end{proposition}
\begin{proof}
	Removing one edge from $Q_n$ does not increase $\gamma_i(Q_n)$  because each square is still independently dominated by a vertex.  However, if we remove two edges from $Q_n$, the independent domination number will change. We have,  $\gamma_i(Q_n-\{e,e_1\})\neq\gamma_i(Q_n)=\lceil \frac{n}{2}\rceil$. Thus, $b_{id}(Q_n)=2$.\qed
\end{proof}

Now we shall compute the id-bondage number of the ortho-chain square cactus graph $O_n$.

\begin{proposition}
	If $O_n$ is an ortho-chain square cactus graph, then $b_{id}(O_n)=2$.
\end{proposition}
\begin{proof}
	It is similar to proof of Proposition \ref{Q_n1}. \qed
\end{proof}

\section{bounds on the id-bondage number of graphs }

In this section, we establish some bounds for the independent domination bondage number of a graph.

We study the relationship between the $id$-bondage of graphs $G$ and $G-e$, where $e\in E(G)$. Also we obtain upper bounds for $b_{id}(G)$. 

 \begin{proposition}\label{minus}
 If $G$ is a graph and $e$ be an edge of $G$, then $$b_{id}(G)\leq   b_{id}(G-e)+1.$$
 \end{proposition}
\begin{proof}
   If $\gamma_i(G)=\gamma_i(G-e)$, then $b_{id}(G)\leq b_{id}(G-e)+1$ and also if $ \gamma_i(G)\neq \gamma_i(G-e)$,   $b_{id}(G)=1$ and so $b_{id}(G)\leq   b_{id}(G-e)+1.$ 
   \qed
\end{proof}

\medskip

 Using Proposition \ref{minus} and mathematical induction, we have 
 \[
 b_{id}(G)\leq   b_{id}(G-e_1-\cdots - e_s)+s,
 \]
 where $1\leq   s \leq   n-2$ and $n=\vert E(G)\vert$. 
\begin{theorem}\label{stab&dist}
	If  $G$ is a graph of order $n$, then $b_{id}(G)\leq  n+1-2\gamma_{i}(G)$.
\end{theorem}
\begin{proof} 
	Set $b_{id}(G)=k$. Thus for every $i$ edges of $G$  ($1\leq   i \leq   k-1$), say $e_1,\ldots, e_{i}$,  we have $\gamma_{id}(G)=\gamma_{id}(G-e_1)=\cdots = \gamma_{id}(G-e_1-\cdots - e_{i})$.  It is known that   the independent domination  number of a graph is at most equal to  half of its order, so $\gamma_i(G)=\gamma_i(G-e_1-\cdots -e_{k-1})\leq   \frac{n-k+1}{2}$. \qed
\end{proof}
\begin{corollary}
	Let $G$ be a graph of order $n\geq 2$. If $b_{id}(G)=n-1$, then $\gamma_i(G)= 1$.
\end{corollary}
\begin{proof} 
	It is a direct consequence of Theorem  \ref{stab&dist}. \qed
\end{proof}

\begin{proposition}\label{q}
	For any graph $G$ with $\gamma_i(G)\ge 2$, we have 
	\[
	b_{id}(G) \leq  \min\{\delta(G)+1, n-\delta(G)-1\}.
	\]
\end{proposition}

 \begin{observation}
 	If $G$ is a graph of order $n$ with $\gamma_i(G)=1$ or $\gamma_i(\overline{G})=1$, then $b_{id}(G)+b_i(\overline{G})\leq n+1$, and this bound is sharp for the complete graphs.
 \end{observation}
 
\begin{theorem}
	If $G$ is a graph of order $n$ with $\gamma_i(G)\ge 2$ and $\gamma_i(\overline{G})\ge 2$, then
	$$b_{id}(G)+b_{id}(\overline{G})\leq\left\{
	\begin{array}{cc}
	n     &\quad   n=2k\\
	n-1   &\quad   n=2k+1
	
	\end{array}\right.
	$$
\end{theorem} 
\begin{proof} 
	If  $n=2k+1$ for some integer $k$, we have $b_{id}(G)+b_{id}(\overline{G})\leq \frac{n-1}{2}+\frac{n-1}{2} =n-1$. Now, we suppose that $n=2k$.
	Using Proposition \ref{q}, we observe that 
	\begin{eqnarray*}
		b_{id}(G)+b_{id}(\overline{G})&\leq& min\{\delta(G) + 1, n-\delta(G)-1\}+ min\{\delta(\overline{G}) + 1, n-\delta(\overline{G})-1\} \\
		&&\leq \frac{n}{2}+\frac{n}{2}=n.
	\end{eqnarray*}\qed
\end{proof}

\section{$id$-bondage of some operations of two graphs}
In this section, we study the independent domination bondage number of some operations of two graphs. First we consider the join of two graphs. 
The join $ G\vee H$ of two graphs $G$ and $H$ with disjoint vertex sets $V(G)$  and edge sets $E(G)$ is the graph union $G\cup H$ together with all the edges joining $V(G)$. 

\begin{observation}\label{join}
	If  $G$ and $H$ are  nonempty graphs, then $$\gamma_i(G\vee H)=\min\{\gamma_i(G),\gamma_i(H)\}.$$
	\end{observation}
	\begin{proof}
		By the definition, every  $\gamma_i$-set  $D$ of  $G$ (or $\gamma_i$-set $D_1$ of  $H$), is a $\gamma_i$-set of $G\vee H$. So we have result.\qed
	\end{proof} 
	
	By Observation \ref{join}, we have the following result.
	
	\begin{theorem}
		If  $G$ and $H$ are  nonempty graphs, then $$b_{id}(G\vee H)=\min\{b_{id}(G),b_{id}(H)\}.$$
	\end{theorem}

Here, we recall the definition of lexicographic product of two graphs.  
 For two graphs $G$ and $H$, let $G[H]$ be the graph with vertex
 set $V(G)\times V(H)$ and such that vertex $(a,x)$ is adjacent to vertex $(b,y)$ if and only if
 $a$ is adjacent to $b$ (in $G$) or $a=b$ and $x$ is adjacent to $y$ (in $H$). The graph $G[H]$ is the
 lexicographic product (or composition) of $G$ and $H$, and can be thought of as the graph arising from $G$ and $H$ by substituting a copy of $H$ for every vertex of $G$ (\cite{DAM}).

 The following theorem gives the independent domination number of $G[H]$. 
 
 \begin{theorem}\label{lexico}
 If  $G$ and $H$ are two  graphs, then $$\gamma_i(G[H]) =\gamma_i(G)\gamma_i(H).$$
 \end{theorem}
\begin{proof}
An independent dominating set in $G[H]$ of minimum cardinality,  arises by choosing an independent dominating
set in $G$ of cardinality $\gamma_i(G)$,  and then, within each copy of $H$ in $G[H]$, 
choosing an independent dominating set in $H$ with cardinality $\gamma_i(H)$. Thus, we have the result. \qed	
	\end{proof} 

By Theorem \ref{lexico}, we have the following result.

\begin{corollary}
 If  $G$ and $H$ are two  graphs, then $b_{id}(G[H])=\min\{b_{id}(G),b_{id}(H)\}$. 
\end{corollary}

Now, we obtain the $id$-bondage of corona of two graphs. We first state and prove the following theorem.
\begin{theorem}\label{corona}
	If  $G$ and $H$ are two  graphs, then $\gamma_i(G\circ H) =|V(G)|\gamma_i(H)$.
\end{theorem}
\begin{proof}
	An independent dominating set in $G\circ H$ of minimum cardinality,  arises by choosing an independent dominating
	set with minimum cardinality in  each copy of $H$ in $G\circ H$.  So we have the result. \qed	
\end{proof} 

By Theorem \ref{corona}, we have the following result.

\begin{corollary}
	If  $G$ and $H$ are two  graphs, then $b_{id}(G\circ H)\leq |V(G)| b_{id}(H)$. 
	\end{corollary}
\begin{proof}
	By Theorem \ref{corona}, $\gamma_i(G\circ H) =|V(G)|\gamma_i(H)$, so by removing $b_{id}(H)$ edges from $H$, the value of   $\gamma_i(G\circ H)$ will be changed. Therefore, we have the result. \qed
	\end{proof}

\end{document}